\documentclass[12pt]{amsart}
\usepackage[utf8]{inputenc}
\usepackage{enumerate,enumitem,amsmath,amssymb,mathrsfs,stmaryrd}
\usepackage{latexsym,epsf,graphicx,comment,appendix}

\newfont{\bb}{msbm10 at 12pt}
\newfont{\tbb}{msbm10 at 8pt}

\def\n{\hbox{\bb N}}

\baselineskip 
\usepackage {amsmath}
\usepackage {amsthm}
\usepackage{times}
\usepackage{amscd}
\usepackage{epsf}

\usepackage[english]{babel}
\usepackage[utf8]{inputenc}
\usepackage{tikz}
\usepackage{color}
\usepackage{times}
\usepackage[T1]{fontenc}
\usepackage[utf8]{inputenc}
\usepackage{amsmath}
\usepackage{dsfont}
\usepackage{amsfonts}

\numberwithin{equation} {section}

\begin{document}
\mbox{}\vspace{0.2cm}\mbox{}

\providecommand{\keywords}[1]
{
  \small	
  \textbf{\textit{Keywords---}} #1
}

\theoremstyle{plain}\newtheorem{lemma}{Lemma}[section]
\theoremstyle{plain}\newtheorem{proposition}{Proposition}[section]
\theoremstyle{plain}\newtheorem{theorem}{Theorem}[section]

\theoremstyle{plain}\newtheorem*{theorem*}{Theorem}
\theoremstyle{plain}\newtheorem*{main theorem}{Main Theorem} 
\theoremstyle{plain}\newtheorem*{lemma*}{Lemma}
\theoremstyle{plain}\newtheorem*{claim}{Claim}
\theoremstyle{plain}\newtheorem{corollary}{Corollary}[section]
\theoremstyle{plain}\newtheorem*{corollary-A}{Corollary}

\theoremstyle{definition}\newtheorem{example}{Example}[section]
\theoremstyle{definition}\newtheorem{remark}{Remark}[section]
\theoremstyle{definition}\newtheorem{definition}{Definition}[section]
\theoremstyle{definition}\newtheorem{acknowledge}{Acknowledgment}
\theoremstyle{definition}\newtheorem{conjecture}{Conjecture}

\begin{center}
\rule{15cm}{1.5pt} \vspace{.4cm}

{\bf\Large Potential theory and applications in conformal geometry} 
\vskip .3cm

Shiguang Ma$\mbox{}^\dag$ and Jie Qing$\mbox{}^\ddag$
\vspace{0.3cm} 
\rule{15cm}{1.5pt}
\end{center}


\title{}

\begin{abstract} In this paper, we want to give an exposition of our recent work on linear and nonlinear 
potential theory and their applications in conformal geometry. We use potential theory to study 
linear and quasilinear equations arising from conformal geometry. We establish the asymptotic behavior near 
singularities and derive applications in conformal geometry.  In particular, we establish some Huber's type theorems 
and Hausdorff dimension estimates of the ends in conformal geometry in general dimensions. 
\end{abstract}


\maketitle

\tableofcontents

\section{Introduction}

This article is written based on the talk delivered at the Second International Conference on Differential Geometry ICDG-FEZ’2024.
The talk was to report some recent results on linear and nonlinear potential theory and their applications in conformal geometry. 

The subject of conformal geometry here in many ways is developed from theories of
analysis and geometry in 2 dimensions. We know that the integral of the Gauss curvature directly gives the Euler number from
the Gauss-Bonnet. Perhaps it is slightly less well known that the differential equation about Gauss curvature facilitated 
the differential geometric approach to the 
uniformization theorem for Riemann surfaces. The theory of supharmonic functions and the Gauss curvature equation in 2 dimensions 
became essential to the theory of open surfaces \cite{Hu57}. The study of Gauss curvature equation also led to the differential
geometric approach to the Teichm\"{u}ller theory of Riemann surfaces \cite{Tromba}. Indeed the Gauss curvature equation 
has been shown 
to be the powerful analytic tool to study the theory of conformal structure on surfaces. Another seminal achievement is exemplified by the 
works on the isospectral problems in 2 dimensions in \cite{OPS1, OPS2, OPS3}. 

To emulate these remarkable achievements in 2 dimensions to higher dimensions, one way is to look for analogs of the Gauss 
curvature equation and analytic tools to study these geometric differential equations. In this article we will review some well known results and report some recent developments. 

We will start in Section 2 with the introduction of generalizations of the Gauss curvature equation in general dimensions 
systematically. Particularly we will introduce a family of $p$-Laplace equations about Ricci curvature tensor as quasilinear generalizations of
the Gauss curvature equation. The introduction of superharmonic as well as 
$p$-superharmonic functions paves the way to the uses of linear and nonlinear potential theories.

In Section 3 we will review motivations in the study of geometric differential equations in conformal geometry. In higher dimensions, it started with 
the Yamabe problem \cite{Sch}.  The Yamabe problem is to search for constant scalar curvature metrics in the hope to find the extension of uniformization theorem in higher dimensions.
From this perspective, we want to provide an overview of current research on geometric differential equations in 
conformal geometry. 

In Section 4 and 5, we will focus on our recent work on applications of potential theory in the study of 
singular solutions of differential equations that are analogs of the Gauss curvature equation in higher dimensions. We then will report
analogs of Huber theorem \cite{Hu57, MQ21a, MQ22g} in higher dimensions, and other geometric and topological consequences in conformal geometry
\cite{MQ22a, LMQZ1, LMQZ2}.  

\section{Partial differential equations in conformal geometry}\label{Sect:PDE}

Let us first quickly introduce curvature on Riemannian manifolds. In a local coordinate $(U, \phi)$ on a Riemannian manifold $(M^n, g)$, 
$$
x=(x_1, x_2, \cdots x_n): \Omega\subset\mathbb{R}^n\to U\subset M^n,
$$
the Riemannian metric and Christoffel symbols are
$$
g_{ij} = g(\partial_{x_i}, \partial_{x_j}) \text{ and } \Gamma_{ij}^k = \frac 12 g^{kl}(\partial_{x_i}g_{jl}
+ \partial_{x_j}g_{il} - \partial_{x_l}g_{ij}).
$$
Therefore the Levi-Civita covariant differentiation is given as
$$
\nabla_{\partial_{x_i}}\partial_{x_j} = \Gamma^k_{ij}\partial_{x_k}
$$
and Riemann curvature tensor is
 $$
R_{ijk}^{\quad l} = - \partial_{x_i}\Gamma^l_{jk} +
\partial_{x_j} \Gamma^l_{ik} - \Gamma^m_{jk}\Gamma^l_{mi} + \Gamma^m_{ik}\Gamma^l_{mj}.
$$
Then Ricci curvature tensor and (Ricci) scalar curvature are
$$
R_{ik} = g^{jl} R_{ijkl} = R_{ijk}^{\quad j} \text{ and } R = g^{ik}R_{ik} = R_i^{\ i}.
$$

In conformal geometry one often encounters the Schouten curvature tensor 
\begin{equation}\label{Equ:schouten}
A_{ik} = \frac 1{n-2}(R_{ik} - \frac 1{2(n-1)} R\, g_{ik})
\end{equation}
and uses the notation $J= g^{ik}A_{ik} = \frac R{2(n-1)}$. 


\subsection{Semilinear equations and superharmonic functions}\label{Subsec:semi-linear}

In dimension 2, let $R_{1212} = K$. Then the celebrated Gauss curvature equation is
\begin{equation}\label{Equ:gauss-curvature-eq}
-\Delta u + K = K_{e^{2u}g}e^{2u}.
\end{equation}
In dimensions greater than 2, the scalar curvature equation is
\begin{equation}\label{Equ:scalar-curvature-eq}
P_2 u = \frac {n-2}{4(n-1)}R_{u^\frac 4{n-2}g}u^\frac {n+2}{n-2}
\end{equation}
where $P_2  = - \Delta  + \frac {n-2}{4(n-1)} R$ is the conformal Laplacian. \eqref{Equ:scalar-curvature-eq} is often referred as the Yamabe
equation and has been extensively studied in connection to the well known Yamabe probelm in conformal geometry.

One may recall that the Paneitz operator is 
$$
P_4  = \Delta^2  + \text{div} (4 A\cdot \nabla - (n-2)J\nabla ) + \frac {n-4}2 Q_4
$$ 
and the associated $Q$-curvature is
$$
Q_4 = - \Delta J + \frac n2 J^2 - 2|A|^2.
$$
The $Q$-equation is 
\begin{equation}\label{Equ:paneitz-eq-4}
P_4 u+ Q_4 = (Q_4)_{e^{2u}g}e^{4u}
\end{equation}
in dimension 4 and
\begin{equation}\label{Equ:paneitz-eq-5}
P_4 u = \frac {n-4}2(Q_4)_{u^\frac 4{n-4}g}u^\frac {n+4}{n-4}
\end{equation}
in dimensions higher than 4. In fact, there are higher order analogs, the so-called GJMS operators $P_{2k}$ and corresponding 
curvature $Q_{2k}$ (cf. \cite{FG85, FG06, GJMS}). They all satisfy the following characteristic properties.
\begin{itemize}
\item The conformal covariance:
$$
(P_{2k})_{u^{\frac 4{n-2k}}g}(\phi) = u^{-\frac {n+2k}{n-2k}}(P_{2k})_{g}(u\cdot\phi) 
$$ 
\item On Euclidean space, $P_{2k}  = (-\Delta)^k$, which is the leading order term in general.
\item The associated scalar curvature of higher order is $Q_{2k} = P_{2k} (1) \text{ when $2k < n$}$.
\end{itemize} 
And
$$
P_{n} u+  Q_n = (Q_n)_{e^{2u}g}e^{nu} \quad \text{ when $n$ is even and}
$$
$$
P_{2k} u = (Q_{2k})_{u^\frac 4{n-2k}g} u^\frac {n+2k}{n-2k} \quad \text{ when $2k < n$.}
$$
Moreover there are fractional analogs coming from the scattering theory of Poincr\'{e}-Einstein manifolds (cf. \cite{GZ03, GQ10}). 
Let $(M^n, [g])$ be the conformal infinity of a Poincar\'{e}-Einstein manifold $(X^{n+1}, g^+)$, where $[g]$ stands for the class of 
conformal metrics on $M^n$. For a representative
$g$, the regularized scattering operator $(P_\alpha)_{g}$ of order $\alpha\in (0, n]$ and the associated nonlocal "curvature"
$$
(Q_\alpha)_{g} = (P_\alpha)_{g} 1
$$
of order $\alpha$ behave similarly 
$$
(P_n)_{g} u + (Q_n)_{g} = (Q_n)_{e^{2u}g} e^{nu} \text{ when $\alpha = n$ is odd, and}
$$
$$
(P_\alpha)_{g} u = (Q_\alpha)_{u^\frac 4{n-\alpha}g} u^\frac {n+\alpha}{n-\alpha} \text{ for $\alpha \in (0, n)$.} 
$$
$P_\alpha$ is the family of fractional powers of Laplacians that includes GJMS operators $P_{2k}$ (cf. \cite{GZ03, mG05, QR1, QR2, GQ10, GQ13}).

We will consider solutions to the equation
\begin{equation}\label{Equ:superhar}
(-\Delta)^\frac \alpha 2 u = \mu \text{ on Euclidean space $\mathbb{R}^n$}
\end{equation}
or  
$$
P_{\alpha} u = \mu \text{ on $M^n$ in general}
$$
for a nonnegative Radon measure $\mu$ as generalized superharmonic functions and appeal to the linear potential theory
in our study of the asymptotic behavior near singularities.


\subsection{$p$-Laplace equations and $p$-superharmonic functions}\label{Subsec:p-laplace}

Here, we want to call the attention to the intermediate Schouten curvature tensor
(cf. \cite{LMQZ1, LMQZ2})
\begin{equation}\label{Equ:intermediate}
A^{(p)} = (p-2)A + J\, g
\end{equation}
for $p\in (1, \infty)$. For $\bar g = u^\frac {4(p-1)}{n-p} g$ and $p\neq n$,
\begin{equation}\label{Equ:transformation}
\aligned
 A^{(p)}_{\bar g}  =  A^{(p)}  - & \frac {2(p-1)}{n-p} \left[ \frac {\Delta u}u g  + (p-2) \frac {D^2u}u \right] \\
  + \frac {2(p-1)}{n-p} & [(1 - (n+p-4)\frac {p-1}{n-p})\frac {|\nabla u|^2}{u^2}g \\
  + (p-2)& ( 1+ \frac {2(p-1)}{n-p}) \frac {\nabla u\otimes \nabla u}{u^2} ].
\endaligned
\end{equation}
Multiplying $u |\nabla u|^{p-2}\frac {u_i}{|\nabla u|}\frac {u_j}{|\nabla u|}$ to and summing up on both sides,
we arrive at the $p$-Laplace equations in conformal geometry
\begin{equation}\label{Equ:p-laplace}
-\Delta_p u + \frac {n-p}{2(p-1)}S^{(p)}(\nabla u) u = \frac {n-p}{2(p-1)}(S^{(p)}(\nabla u))_{\bar g} u^q
\end{equation}
where 
$$
S^{(p)}(\nabla u) = |\nabla u|^{p-2} A^{(p)}(\nabla u), \quad  q = \frac {2p(p-1)}{n-p} + 1,
$$
and $A^{(p)}(\nabla u)$ is the $A^{(p)}$ curvature in the direction $\nabla u$. 
When $p=n$, we realize $A^{(n)} = Ric$, and the $n$-Laplace equation is
\begin{equation}\label{Equ:n-laplace}
-\Delta_n\phi + |\nabla\phi|^{n-2}Ric(\nabla\phi) = (|\nabla\phi|^{n-2}Ric(\nabla\phi))_{\bar g} e^{n\phi}
\end{equation}
where $\bar g = e^{2\phi}g$. Recall
$$
\Delta_p u = \text{div}(|\nabla u|^{p-2}\nabla u).
$$
For $p=n=2$, \eqref{Equ:n-laplace} goes back to the Gauss curvature equation
$$
-\Delta \phi + K = K_{\bar g}e^{2\phi},
$$
where $\bar g = e^{2\phi}g$. 
For $p=2$ and $n\geq 3$, the intermediate Schouten curvature goes back to the scalar curvature and the $p$-Laplace equation
\eqref{Equ:p-laplace} goes back to the scalar curvature equation
$$
-\Delta u + \frac {n-2}{4(n-1)}R u = \frac {n-2}{4(n-1)} R_{\bar g} u^\frac {n+2}{n-2},
$$
where $\bar g = u^\frac 4{n-2} g$. 
For $p > n$, the $p$-Laplace equation \eqref{Equ:p-laplace} for the intermediate Schouten curvature is still valid
for $\bar g = u^{- \frac {4(p-1)}{p-n}} g$ and $q= - \frac {2p(p-1)}{p-n} + 1 < 0$. And, when taking $p\to\infty$, we
arrive at the infinite Laplace equation on Schouten curvature $A$
\begin{equation}\label{Equ:infinite-laplace}
- \Delta_\infty u - \frac 12 |\nabla u|^2 A(\nabla u) u = -\frac 12 (|\nabla u|^2 A(\nabla u))_{\bar g} u^{-7}
\end{equation}
for $\bar g = u^{-4}\, g$. Recall $\Delta_\infty u = u_{ij}u_iu_j.$

We will consider solutions to the equation
\begin{equation}\label{Equ:p-superhar}
-\Delta_p u = \mu
\end{equation}
for a nonnegative Radon measure $\mu$ as $p$-superharmonic functions and appeal to nonlinear potential theory
in our study of the asymptotic behavior near singularities.


\section{Motivating problems in conformal geometry}\label{Sect:motivating}

\subsection{Motivations from the surface theory}\label{Subsec:surfaces}
Much of the developments in conformal geometry have been motivated by the tremendous success in the study of surfaces. We want to start with
the uniformization theorem (Klein 1883 Poincar\'{e} 1882 Koebe 1907), which states, the universal covering of a closed Riemann surface is conformally equivalent to  one of the three: $\mathbb{D}^2$, $\mathbb{R}^2$, $\mathbb{S}^2$. 

By the differential-geometric approach, this is to say that, on a closed Riemannian surface $(M^2, g)$, there is a conformal metric $e^{2u}g$
whose Gauss curvature is constant. Equivalently, one solves 
$$
-\Delta u + K =\kappa e^{2u}
$$
for $\kappa = -1, 0, 1$.
 
On the theory of open surfaces, there is the important theorem,
proved by Huber in \cite{Hu57}, which states, if a complete
noncompact surface $(M,g)$ satisfies 
\[
\int_{M}K^{-}d\mu_{g}<+\infty,
\]
then it is conformally equivalent to a compact Riemann surface with
finitely many points removed. Here $K^{-}=\max\{- K,0\}$ is the negative
part of Gaussian curvature $K$ and $d\mu_{g}$ is the volume form
of the metric $g$. Analytically this includes the statement: 

\begin{theorem*}
Let the conformal metric $\bar g= e^{2u}g$ on $\Omega\setminus S$ be geodesically complete near $S$. Suppose that 
$\int_{\Omega\setminus S} (K^-dvol)_{\bar g} <\infty$. Then $S$ consists of at most finitely many points.
\end{theorem*}


\subsection{Interplays of analysis and conformal geometry in higher dimensions} 
There have been many interesting work to establish theorems of 
Huber-type in higher dimensions. 
To motivate, we want to make two remarks.
First, in two dimensions, open surfaces are well understood topologically. However, 
in higher dimensions, the topology as well as the local and global conformal structure 
become much more complicated. Secondly, in higher dimensions, Ricci curvature tensor 
is genuinely no longer a scalar. Therefore, for applications in geometry, we will focus on 
locally conformally flat manifolds, particularly those whose development maps are injective.  

One may focus on locally conformally flat manifolds as they are more closely related to theories of surfaces. 
Suppose that $(M^n, g)$ is a locally conformally flat manifold and that the conformal immersion from a covering
$(\tilde M^n, \tilde g)$ to $(\mathbb{S}^n, g_{\mathbb{S}})$ is injective.

\hskip 1.6in\begin{tikzpicture}
    \node (x) at (0,0) {$\tilde M^n$};
    \node (y) at (5,0) {$\mathbb{S}^n$};
    \node (z) at (0, -2) {$M^n$};    
    \path (x) edge node [above] {$\phi$} (y);
    \path (x) edge node [left] {$\pi$} (z);
\end{tikzpicture}
\vskip 0.1in

Then, on $\phi(\tilde M^n)\subset \mathbb{S}^n$, there is a complete conformal metric $(\phi^{-1})^*\tilde g = e^{2u} g_{\mathbb{S}}$. One is
interested in the size of $\mathbb{S}^n\setminus\phi(\tilde M^n)$, or specifically, the Hausdorff dimension of
$\mathbb{S}^n\setminus\phi(\tilde M^n)$. We want to confirm that, in these situations, there are strong correlations between the Hausdorff dimension
of $\mathbb{S}^n\setminus\phi(\tilde M^n)$, the topology of $M^n$, and the curvature of the metric $g$.  
Specifically,  we will introduce notions of the positivity of curvature tensors and their impact to the homology groups and
homotopy groups.

One way to measure the positivity of curvature and derive topological consequences is to consider (scalar) curvatures of higher order and nonlinear 
nature and the associated curvature equations like the semilinear and quasilinear equations mentioned in the previous section in the spirit of the
uniformization theorem on surfaces. The Yamabe problem and its generalizations have been introduced and generated huge interest, which has been one of the most significant and influential subjects in geometric analysis and geometric partial differential equations. We believe the newly introduced 
$p$-Laplace equations in conformal geometry will add fuel and contribute further developments.

For instance, as a way to gauge the positivity of Ricci curvature tensor, we propose here to use the cones 
\begin{equation}\label{Equ:A-p-cone}
\mathcal{A}^{(p)}=\{\lambda\in \mathbb{R}^n: \min_k\{(p-2)\lambda_k + \sum_{i=1}^n \lambda_i\}\geq 0\}
\end{equation}
for $p\in (1, \infty)$ to describe the positivity of the intermediate Schouten curvature tensor $A^{(p)}$ when $\{\lambda_i\}$ stands for the
eigenvalues of the Schouten curvature tensor $A$. 

To illustrate how effective this approach can lead to vanishing theorems on topology,
we recall the cones
\begin{equation}\label{Equ:R-r-cone}
\mathcal{R}^{(r)} =  \{\lambda\in \mathbb{R}^n: \min \{(n-r)\sum_{k=1}^r \lambda_{i_k} + r \sum_{k=r+1}^n \lambda_{i_k}\}\geq 0\}
\end{equation}
introduced for dealing with the terms in the Bochner formula on $r$-forms on locally conformally flat $n$-manifolds, where $\min$ is taken for
all possible re-arrangements $\{\lambda_{i_1}, \lambda_{i_2}, \cdots, \lambda_{i_n}\}$ of $(\lambda_1, \lambda_2, \cdots, \lambda_n)$. 
We observe that

\begin{lemma} (\cite{LMQZ1})
$$
\mathcal{A}^{(p)} \subset \mathcal{A}^{(q)} \quad \text{ for $2 \leq q < p < \infty$}
$$ 
$$
\mathcal{R}^{(s)} \subset \mathcal{R}^{(r)} \quad \text{ for $0 < s \leq r \leq \frac n2$}
$$
and
$$
\mathcal{A}^{(p)}  \subset \mathcal{R}^{(r)} \quad \text{ for $\frac {n-p}2 + 1\leq r \leq \frac n2$}.
$$
Obviously $\mathcal{A}^{(2)} = \mathcal{R}^{(\frac n2)}$ is the baseline.
\end{lemma}
Consequently, we obtain

\begin{theorem} (\cite{LMQZ1})
Let $(M^n, g)$ be a compact locally conformally flat manifold with $A^{(p)}\geq 0$ for $p\in [2, n)$. 
Suppose that the scalar curvature is positive somewhere on $M^n$. Then, for $\frac {n-p}2 +1\leq k\leq \frac {n+p}2-1$, the 
Betti numbers $\beta_k=0$, unless $(\tilde{M^n}, g) \overset{\text{isometric}}{\sim} \mathbb{H}^k\times \mathbb{S}^{n-k}$.
\end{theorem}

The proof uses The Bochner formula on $r$-forms (cf. \cite{GLW, Nay97})
$$
\Delta \omega = \nabla^*\nabla \omega + \mathcal{R}(\omega)
$$
$$\mathcal{R}(\omega) = ((n-r)\sum_{i=1}^r\lambda_i + r\sum_{i=r+1}^n\lambda_i)\omega$$
for $\omega = \omega_1\wedge\omega_2\cdots\wedge\omega_r$ and $\{\omega_k\}$ is the orthonormal basis under which the 
Schouten curvature tensor $A$ is diagonalized on locally conformally flat $n$-manifolds.

In the light of \eqref{Equ:p-laplace}, we therefore want to use asymptotic behavior $p$-superharmonic functions to estimate the
size of singularities and derive consequences on the homotopy groups when assuming $A^{(p)} \geq 0$ for $p\in [2, \infty)$.
 
 
 \section{Linear potentials and applications in conformal geometry}\label{Sect:Riesz}
 
In this section, we give an exposition of our recent work on linear potential theory in conformal geometry. We apply linear potential theory to study partial differential equations arising from conformal geometry and, in particular, the problems related to the dimension of the boundary of a domain
that admits certain complete conformal metric. 

\subsection{Riesz potential, capacity, and thin set}\label{Subsec:linear-potential}

Let $\Omega$ be a bounded open subset in the Euclidean space $\mathbb{R}^n$. Then, for $x\in \Omega$, let 
\begin{equation}\label{Equ:potential-euclidean}
R^{\alpha, \Omega}_\mu (x) = \left\{\aligned \int_\Omega \frac 1{|x-y|^{n-\alpha}} d\mu(y)  \quad & \text{ when 
$\alpha \in (1, n)$}\\ \int_\Omega \log \frac D{|x-y|} d\mu(y) \quad & \text{ when $\alpha = n$}
\endaligned\right.
\end{equation}
for a Radon measure $\mu$ on $\Omega$, where $D$ is the diameter of $\Omega$. Let $E$ be a subset in $\Omega$ and $\Omega$ be a bounded
open subset in $\mathbb{R}^n$. For $\alpha\in (1, n]$,  we define the Riesz capacity by
\begin{equation}\label{Equ:linear-capacity}
C^\alpha_{\mathcal{L}} (E, \Omega) = \inf \{ \mu(\Omega): \quad \aligned \mu \geq 0 & \text{ Radon measure on $\Omega$}\\
R^{\alpha, \Omega}_\mu & (x) \geq 1  \text{ for all $x\in E$}\endaligned\}.
\end{equation}
The following basic properties are easy to prove (cf. \cite{Mi96})
\begin{lemma}\label{Lem:l-basic property}

Let $C^\alpha_{\mathcal{L}}$ be the Riesz capacity defined as in \eqref{Equ:linear-capacity} for $\alpha\in(1,n]$. Then 
\begin{enumerate}
\item $C^{\alpha}_{\mathcal{L}}$ is nondecreasing in $E$, that is 
\[
C^{\alpha}_{\mathcal{L}}(E_{1},\Omega)\le C^{\alpha}_{\mathcal{L}}(E_{2},\Omega)
\]
 when $E_{1}\subset E_{2}\subset\Omega\subset\mathbb{R}^{n}$. 
\item $C^{\alpha}_{\mathcal{L}}$ is countably subadditive, that is 
\[
C^{\alpha}_{\mathcal{L}}(\cup_{i=1}^{\infty}E_{i},\Omega)\le\sum_{i=1}^{\infty}C^{\alpha}_{\mathcal{L}}(E_{i},\Omega)
\]
for subsets $E_{i}\subset\Omega$.
\item For a positive number $\lambda$, let
\[
A_{\lambda}=\{\lambda x:x\in A\}
\]
for any subset $A$ of $\mathbb{R}^{n}.$ Then, for $\alpha\in(1,n]$,
\[
C^{\alpha}_{\mathcal{L}}(E_{\lambda},\Omega_{\lambda})=\lambda^{n-\alpha}C^{\alpha}_{\mathcal{L}}(E,\Omega).
\]
\item Suppose 
\[
\Phi:\Omega\to\Omega
\]
 is a contractive map, that is
\[
|\Phi(x)-\Phi(y)|\le|x-y|
\]
 for all $x,y\in\Omega$. Then, for $\alpha\in(1,n],$
\[
C^{\alpha}_{\mathcal{L}}(\Phi(E),\Omega)\le C^{\alpha}_{\mathcal{L}}(E,\Omega)
\]
for any subset $E\subset\Omega$.
\item For $\alpha\in(1,n],$
\[
C^{\alpha}_{\mathcal{L}}(\partial B_{1}(0),B_{2}(0))=c(n,\alpha)
\]
for some positive constant $c(n,\alpha).$
\end{enumerate}
\end{lemma}

Thin sets with respect to the Riesz capacity $C^\alpha_{\mathcal{L}}$ are defined through dyadic annuli:
$$
\omega^\delta_i (p) = \{ x\in \mathbb{R}^n: |x - p| \in [2^{-i}\delta,  2^{-i+1}\delta]\}
$$
$$
\Omega^\delta_i(p) = \{x\in \mathbb{R}^n: |x - p| \in (2^{-i-1}\delta , 2^{-i+ 2} \delta)\}.
$$

\begin{definition}\label{Def:l-thin}
Let $E$ be a subset in the Euclidean space $\mathbb{R}^n$ and $p\in \mathbb{R}^n$. 
The subset $E$ is said to be $\alpha$-thin with respect to the Riesz capacity $C^{\alpha}_{\mathcal{L}}$ at the point $p$ for $\alpha \in (1, n)$ if 
$$
\sum_{i\geq 1} \frac { C^{\alpha}_{\mathcal{L}} (E\bigcap \omega^\delta_i (p), \Omega^\delta_i (p))}
{C^{\alpha}_{\mathcal{L}} (\partial B_{2^{-i}\delta} (p), B_{2^{-i+1}\delta}(p))}
< \infty
$$  
for some small $\delta>0$.  The subset $E$ is said to be $n$-thin at $p$ if 
$$
\sum_{i\geq 1} i C^{n}_{\mathcal{L}} (E\bigcap \omega^\delta_i(p), \Omega^\delta_i(p)) < \infty
$$
for some small $\delta > 0$. 
\end{definition}

A simple and important fact (see also \cite{Mi96}) is 
\begin{lemma}\label{Lem:l-thin-consequence}
Let $E$ be a subset in the Euclidean space $\mathbb{R}^n$ and $p\in \mathbb{R}^n$. Suppose that $E$ is $\alpha$-thin at the 
point $p$ for $\alpha \in (1, n]$. Then there is always a ray from $p$ that avoids $E$ at least within some small ball at $p$.
\end{lemma}
 
 
 \subsection{Singularities of superharmonic functions under the fine topology}\label{Subsec:Riesz-estimates}

What we want to derive is the asymptotic behavior of superharmonic functions near singularities outside thin sets.
\begin{theorem}\label{Thm:l-asym} 
Suppose $\mu$ be a finite nonnegative Radon measure on a bounded domain $\Omega\subset R^n$.  
Then, for $p\in \Omega$, where $R^{\alpha, \Omega}_\mu(p) = \infty$, there is a subset $A$ that is $\alpha$-thin at $p$ such that
$$
\lim_{x \to p \text{ and } x \in R^n\setminus A} \frac {R^{\alpha, \Omega}_\mu (x)}{|x-p|^{\alpha - n}} = \mu(\{p\}).
$$
for $\alpha\in (1, n)$ and
$$
\lim_{x \to p \text{ and } x \in R^n\setminus A} \frac {R^{n, \Omega}_\mu (x)}{\log \frac 1{|x - p|}} = \mu(\{p\})
$$
for $\alpha = n$.
\end{theorem}
 
 To obtain the Hausdorff dimension estimates we use the help from the following
 
 \begin{lemma} \label{Lem:Kpata} (\cite{Kp})
Let $\mu$ be a nonnegative Radon measure on a complete Riemannian manifold $(M^n, g)$ and let
$$
G_d^\infty = \{x\in M^n: \limsup_{r\to 0} r^{-d} \mu(B_r(x)) = +\infty\}
$$
for any $d\in [0, n]$.
Then 
$$
\mathcal {H}_d (G^\infty_d) = 0
$$
where $\mathcal{H}_d$ is the Hausdorff measure of dimension $d$.
\end{lemma}

Consequently we have

 \begin{theorem} \label{Thm:l-improve} (\cite{MQ22a}
Suppose that $\mu$ is a finite nonnegative Radon measure on a bounded 
domain $\Omega\subset\mathbb{R}^n$. Let $S$ be a compact subset in $\Omega$ such that its Hausdorff dimension is greater than $d$, where
 $d < n-\alpha$ and $\alpha \in (1, n)$.  Then there is a point $p\in S$ and a subset $E$ that is $\alpha$-thin at $p$ such that 
$$
R_\mu^{\alpha, \Omega} (x) \leq \frac C{|x - p|^{n-\alpha - d}}
$$
for some constant $C$ and all $x\in B_\delta (p) \setminus E$ for some small $\delta > 0$.
\end{theorem}


\subsection{Consequences in conformal geometry}\label{Subsec:l-geometric} 
We first derive the Hausdorff dimension estimates from the scalar curvature equation \eqref{Equ:scalar-curvature-eq}.

\begin{theorem}\label{Thm:scalar>=3} (\cite{MQ22a})
Let $S$ be a compact subset and $D$ be a bounded open neighborhood of $S$.  
Suppose that $\bar g = u^\frac 4{n-2} g$ is a conformal metric on $D\setminus S$ and is geodesically complete near $S$. 
Then the Hausdorff dimension
$$
\text{dim}_\mathcal{H} (S) \leq \frac {n-2}2,
$$ 
provided that
$R^-[\bar g] \in L^\frac {2n}{n+2} (D\setminus S, \bar g) \bigcap L^p(D\setminus S, \bar g)$
for some $p > n/2$.
\end{theorem}
Obviously the integrability condition holds if $R\geq 0$. Therefore this improves some early result of Schoen-Yau 1988 (\cite{SY88}).
Using \eqref{Equ:paneitz-eq-5} we have

\begin{theorem} \label{Thm:Q>=5}\cite{MQ22a}
Let $S$ be a compact subset and $D$ be a bounded open neighborhood of $S$. Suppose that $\bar g = u^\frac 4{n-4} g$ is a 
conformal metric on $D\setminus S$ with nonnegative scalar curvature and is geodesically complete 
near $S$. And suppose also that
$$
(Q_4^-)_{\bar g} \in L^\frac {2n}{n+4}(D\setminus S, \bar g).
$$
Then 
$$
\text{dim}_{\mathcal{H}} (S) \leq \frac {n-4}2.
$$
\end{theorem}
We also were able to show a Huber's type theorem in dimension 4.

\begin{theorem}\label{Thm:huber-4} \cite{MQ22g})
Let $S$ be a compact subset and $D$ be a bounded open neighborhood of $S$. Suppose that $\bar g= e^{2u} g$ 
is a conformal metric on $D\setminus S$ with nonnegative scalar curvature and is geodesically complete near $S$. 
And suppose that
$$
\int_D (Q_4^-dvol)_{\bar g} < \infty.
$$
Then $S$ consists of at most finitely many points.
\end{theorem}
This can be compared with some early result of Chang-Q-Yang 2000 (\cite{CQY}).


\section{Nonlinear potentials and applications}\label{Sect:Wolff}

In this section we give an exposition of our recent work on nonlinear potential theory in conformal geometry. We apply nonlinear potential theory to study 
$p$-Laplace equations arising from conformal geometry and, in particular, the problems related to the asymptotic behavior near 
and the size of singularities in conformal geometry.


\subsection{Wolff potentials, $p$-superharmonic functions, and their singular behavior}\label{Subsec:nonlinear-potential}

Let us first give the definitions of $p$-harmonic and  $p$-superharmonic functions. In this section we always assume $1<p<n$ unless specified otherwise.
Let us recall again what is the $p$-Laplace operator
\begin{equation*}
\Delta_p u = \text{div}(|\nabla u|^{p-2}\nabla u).
\end{equation*}

\begin{definition}\label{Def:p-harmonic} (\cite[Definition 2.5]{Lind06})
We say that $u\in W^{1,p}_{\rm{loc}}(\Omega)$ is a weak solution of the $p$-harmonic equation in $\Omega$, if 
\begin{equation*}
\int \langle |\nabla u|^{p-2}\nabla u,\nabla \eta \rangle dx=0
\end{equation*}
for each $\eta\in C_0 ^{\infty}(\Omega)$. If, in addition, $u$ is continuous, we then say $u$ is a $p$-harmonic function.
\end{definition}

\begin{definition}\label{Def:p-superhar} (\cite[Definiton 5.1]{Lind06})
A function $v:\Omega\to (-\infty,\infty]$ is called $p$-superharmonic in $\Omega$, if 
\begin{itemize}
\item $v$ is lower semi-continuous in $\Omega$;
\item $v\not\equiv\infty$ in $\Omega$;
\item for each domain $D\subset\subset\Omega$ the comparison principle holds, that is, 
if $h\in C(\bar{D})$ is $p$-harmonic in $D$ and $h|_{\partial D}\le v|_{\partial D}$, then $h\le v$ in D. 
\end{itemize}
\end{definition}

As stated in \cite[Theorem 2.1]{KM92}, for $u$ to be a $p$-superharmonic function in $\Omega$, there is a nonnegative Radon 
measure $\mu$ in $\Omega$ such that
\begin{equation}\label{Equ:p-laplace-measure}
-\Delta_p u=\mu.
\end{equation}

For more about $p$-superharmonic functions and nonlinear potential theory, we refer to  \cite{KM94, HKM93, HK88, AH99, Lind06} and
references therein. The most important tool is the following Wolff potential 

\begin{equation}\label{Equ:wolff-potential}
W^{\mu}_{1,p}(x,\, r)=\int_0^r(\frac {\mu(B(x,\, t))}{t^{n-p}})^{\frac 1 {p-1}}\frac {dt}{t}
\end{equation}
for any nonnegative Radon measure $\mu$ and $p\in (1, n]$. 
The fundamental estimate for the use of the Wolff potential in the study of $p$-superharmonic
functions is as follows:

\begin{theorem}\label{Thm:main-use-wolff} (\cite[Theorem 1.6]{KM94}) 
Suppose that $u$  is a nonnegative $p$-superharmonic function satisfying \eqref{Equ:p-laplace-measure} for a nonnegative 
finite Radon measure $\mu$ in $B(x, 3r)$. Then
\begin{equation}\label{Equ: main-use-wolff}
c_1 W^{\mu}_{1,p}(x, r)\le u(x)\le c_2(\inf_{B(x, r)}u+W^{\mu}_{1,p}(x, 2r))
\end{equation}
for some constants $c_1(n,p)$ and $c_2(n,p)$ for $p\in (1, n]$.
\end{theorem}

To introduce nonlinear potential theory and present the results on asymptotic behavior of the Wolff potentials at singularities we first recall 
some definitions and basics.
\begin{definition}\label{Def:p-capacity} (\cite[Section 3.1]{KM94}) 
For a compact subset $K$ of a domain $\Omega$ in $\mathbb{R}^n$, we define
\begin{equation}\label{Equ:p-capacity-K}
cap_p(K,\Omega)=\inf \{\int_{\Omega}|\nabla u|^p dx: \text{ $u\in C_0^{\infty}(\Omega)$ and $u\ge 1$ on $K$}\}.
\end{equation}
Then $p$-capacity for arbitrary subset $E$ of $\Omega$ is 
\begin{equation}\label{Equ:p-capacity}
cap_p(E,\Omega)=\inf_{
\text{open}\ G \supset E \text{ \& } G  \subset\Omega} \ \ \sup_{\text{compact} \ K \subset G} cap_p(K,\Omega).
\end{equation}
\end{definition}

Analogously, the capacity $cap_p$ shares the same basic properties in Lemma \ref{Lem:l-basic property} as the Riesz capacity does
(cf. \cite{HKM93, AH99, LMQZ1}). The notions of thinness in the potential theory are vitally important. 
The readers are referred to \cite{AM72, AH73, HK88,KM94, MQ21a, MQ22a, MQ22g} for detailed discussions and references therein. 
To study the singular behavior of $p$-superharmonic functions, like \cite[Definition 3.1]{MQ21a} (see also
\cite[Section 2.5]{Mi96}), we 
propose a thinness that is less restrictive than that by the Wiener type integral when $p\in (2, n)$. Recall the dyadic annuli are
\begin{eqnarray*}
\omega_i(x_0)&=&\{x\in\mathbb R^n:2^{-i}\le|x-x_0|\le2^{-i+1}\};\\
\Omega_i(x_0)&=&\{x\in\mathbb R^n:2^{-i-1}\le|x-x_0|\le2^{-i+2}\}.
\end{eqnarray*}

\begin{definition}\label{Def:quasi-p-thin}
A set $E\subset\mathbb R^n$ is said to be $p$-thin for singular behavior for $p\in (1, n)$ at $x_0\in\mathbb R^n$ if 
\begin{equation}\label{Equ:quasi-p-thin}
\sum_{i=1}^\infty \frac{cap_p(E\cap\omega_i(x_0),\Omega_i(x_0))}{cap_p(\partial B(x_0, 2^{-i}), B(x_0, 2^{-i+1}))}  <+\infty.
\end{equation}
Meanwhile a set $E$ is said to be $n$-thin at $x_0\in \mathbb{R}^n$ if
\begin{equation}\label{Equ:quasi-n-thin}
\sum_{i=1}^\infty i^{n-1}\, cap_n(E\cap\omega_i (x_0), \Omega_i(x_0)) < +\infty.
\end{equation} 
\end{definition}

Notice that these thin sets are with respect to the $p$-capacity $cap_p$ in contrast to the ones with respect to the Riesz capacity in 
Definition \ref{Def:l-thin}. Like Lemma \ref{Lem:l-thin-consequence}, the most important fact about $p$-thinness for singular behavior to us due to Lemma \ref{Lem:l-basic property} is the following 

\begin{lemma}\label{Lemsegment-escape} 
Let $E$ be a subset in the Euclidean space $\mathbb{R}^n$ and $x_0\in \mathbb{R}^n$ be a point. Suppose that $E$ is $p$-thin for singular behavior at the 
point $x_0$ for $p\in (1, n]$. Then there is a ray from $x_0$ that avoids $E$ at least within some small ball at $x_0$.
\end{lemma}

The key estimates on the Wolff potential are the following

\begin{theorem} (\cite{LMQZ1, LMQZ2}) \label{Thm:wolff potential upper bdd}
Suppose $\mu$ is a nonnegative finite Radon measure in $\Omega$. Assume that, for a point $x_{0}\in\Omega$ and some number $m\in (0, n-p)$,
\begin{equation}\label{Equ:lebesgue-p}
\mu(B(x_{0}, t))\le Ct^{m}
\end{equation}
for all $t\in (0, 3r_0)$ with $B(x_0, 3r_0)\subset \Omega$. Then, for $\varepsilon>0$, there are a subset $E\subset\Omega$, which
is $p$-thin for singular behavior at $x_{0}$, and a constant $C>0$ such that
\begin{equation}\label{Equ:wolff-upper}
W_{1,p}^{\mu}(x, \, r_0)\le C|x - x_0|^{-\frac{n-p-m +\varepsilon}{p-1}} \text{ for all $x\in\Omega\setminus E$}
\end{equation}
for $p\in [2, n)$.
\end{theorem}

And

\begin{theorem} (\cite{LMQZ1, LMQZ2}) \label{Thm:wolff potential upper bdd-singualarity}
Let $\mu$ be a nonnegative finite Radon measure in $\Omega$ and $B(x_0, 3r_0)\subset \Omega$. 
Then there is a subset $E$ that is $p$-thin for the singular behavior at $x_0$ such that
\begin{equation*}
\lim_{x\to x_0 \text{ and } x\notin E } |x-x_0|^\frac {n-p}{p-1} W_{1,p}^{\mu}(x, \, r_0) = \frac{p-1}{n-p} \mu(\{x_0\})^\frac 1{p-1}
\end{equation*}
for $p\in (1, n)$. Similarly, there is a subset $E$ that is $n$-thin for the singular behavior at $x_0$ such that
\begin{equation*}
\lim_{x\to x_0 \text{ and } x\notin E} \frac{W_{1,n}^{\mu}(x, \, r_0)}{\log\frac 1{|x-x_0|}} = \mu(\{x_0\})^\frac 1{n-1}.
\end{equation*}
\end{theorem}

It turns out the notion of thinness indeed facilitates the fine topology (cf. \cite{KM94, Mi96}) in the following context, which gives the equivalent
analytic definition of thinness. 

\begin{theorem} (\cite{LMQZ1, LMQZ2}) \label{Thm:converse-p-thin} Suppose that $E$ is a subset that is $p$-thin for the singular behavior at the origin
according to Definition \ref{Def:quasi-p-thin} for $p\in (1, n]$. And suppose that the origin is in $\bar E\setminus E$. Then, when $p\in (1, n)$, 
there is a Radon measure $\mu$ in a neighborhood of the origin such that, for some fixed $r_0>0$, 
$$
\lim_{x\to 0 \text{ and } x\in E} |x|^\frac {n-p}{p-1}W^\mu_{1, p}(x, r_0) = \infty.
$$
Similarly, when $p=n$, there is a Radon measure $\mu$ in a neighborhood of the origin such that, for some fixed $r_0>0$, 
$$
\lim_{x\to 0 \text{ and } x\in E} \frac {W^\mu_{1, n}(x, r_0)}{\log\frac 1{|x|}} = \infty.
$$
\end{theorem}

In fact, we use Theorem \ref{Thm:main-use-wolff} to derive the estimates for $p$-superharmonic functions. From Theorem 
\ref{Thm:wolff potential upper bdd-singualarity}, we have

\begin{theorem} (\cite{LMQZ2}) \label{Thm:main-asymptotic-p}
Suppose that $u$ is a nonnagetive $p$-superharmonic function in $\Omega\subset\mathbb{R}^n$ satisfying
$$
-\Delta_p u = \mu \text{ in $\Omega$}
$$
for a nonnegative finite Radon measure $\mu$ on $\Omega$ and $p\in (1, n]$. Then, for $x_0\in\Omega$, there is a subset $E$ that is $p$-thin for singular
behavior at $x_0$ such that
\begin{equation}\label{Equ:p-super-asymptotic}
\lim_{x\to x_0\text{ and } x\notin E}  \frac{u(x)}{G_p(x, x_0)} = m = \left\{\aligned \frac{p-1}{n-p} (\frac {\mu(\{0\})}{|\mathbb{S}^{n-1}|})^\frac 1{p-1}
& \text{ when $p\in (1, n)$}\\ (\frac {\mu(\{0\})}{|\mathbb{S}^{n-1}|})^\frac 1{n-1} & \text{ when $p=n$}\endaligned\right.,
\end{equation} 
where
\begin{equation}\label{Equ:G-p}
G_p(x, x_0) =\left\{\aligned |x-x_0|^{-\frac {n-p}{p-1}} & \text{ when } p\in (1, n)\\
- \log |x-x_0| & \text{ when } p=n\endaligned\right..
\end{equation}
Moreover $u(x) \geq mG_p(x, x_0) - c_0$ for some $c_0$ and all $x$ in a neighborhood of $x_0$.
\end{theorem}

From Theorem \ref{Thm:converse-p-thin}, we have

\begin{theorem} (\cite{LMQZ2}) \label{Thm:infty-p-intr}
Suppose that $E$ is a subset that is $p$-thin for the singular behavior at the origin
according to Definition \ref{Def:quasi-p-thin} for $p\in (1, n]$. And suppose that the origin is in $\bar E\setminus E$. Then, when $p\in (1, n)$, 
there is a $p$-superharmonic function $u$ in a neighborhood of the origin such that
$$
\lim_{x\to 0 \text{ and } x\in E} |x|^\frac {n-p}{p-1} u(x)  = \infty.
$$
Similarly, when $p=n$, there is a $n$-superharmonic function $u$ in a neighborhood of the origin such that
$$
\lim_{x\to 0 \text{ and } x\in E} \frac {u(x)}{\log\frac 1{|x|}} = \infty.
$$
\end{theorem}


\subsection{Consequences in conformal geometry}\label{Subsec:geometric-consequences}

In this subsection we will use the $p$-Laplace equation \eqref{Equ:p-laplace} and Theorem \ref{Thm:wolff potential upper bdd} to derive the
consequence of the curvature condition $A^{(p)}\geq 0$. In the light of Lemma \ref{Lem:Kpata} from Subsection \ref{Subsec:Riesz-estimates},
we apply Theorem \ref{Thm:wolff potential upper bdd} to prove

\begin{theorem} (\cite{LMQZ1, LMQZ2}) \label{Thm:main-geometric-thm-1}
Suppose that $S$ is a compact subset of a bounded domain $\Omega\subset\mathbb{R}^n$.  And suppose that there is a metric $\bar g$ on 
$\Omega\setminus S$ such that 
\begin{itemize}
\item it is conformal to the Euclidean metric $g_{\mathbb{E}}$;
\item it is geodesically complete near $S$. 
\end{itemize}
Assume that $A^{(p)} [\bar g]\geq 0$ for some $p\in [2, n)$.
Then
$$
dim_{\mathcal{H}}(S) \leq \frac {n-p}2
$$
for $p\in [2, n)$.
\end{theorem}

As a consequence,

\begin{theorem} (\cite{LMQZ1}) \label{Thm:main-geometric-thm-2} Suppose that $S$ is a closed subset of the sphere $\mathbb{S}^n$.  
And suppose that there is a metric $\bar g$ on $\mathbb{S}^n\setminus S$ that is conformal to the standard round metric 
$g_{\mathbb{S}}$. Assume that it is geodesically complete near $S$ and that $A^{(p)} [\bar g]\geq 0$ for some $p\in [2, n)$.
Then
$$
dim_{\mathcal{H}}(S) \leq \frac {n-p}2.
$$
\end{theorem}

Historically, in \cite{BMQ-r, MQ21a, MQ22g}, $n$-Laplace equations and applications in hypersurfaces and conformal geometry were first investigated
in connection to find Huber's type theorems in general dimensions. The study of asymptotic behavior of $n$-superharmonic functions at singularities
was also carried out by using on nonlinear potential theory in \cite{LMQZ1, LMQZ2} (cf. Theorem \ref{Thm:main-asymptotic-p} in previous section).

\begin{theorem} (\cite{MQ22g}) \label{Thm:intro-2}
For $n\geq 3$, let $D$ be a bounded domain in the Euclidean space $(\mathbb{R}^n, |dx|^2)$ and let $S\subset D$ be a subset which is closed in $\mathbb R^n$.
Suppose that, on $D\backslash S$, there is a conformal metric $g= e^{2v}|dx|^2$ satisfying
$$
\lim_{x\to S}v(x)=+\infty \text{ and } 
Ric^{-}_g|\nabla v|^{n-2}e^{2v} \in L^{1}(D\backslash S,|dx|^2).
$$
Then $S$ is a finite point set. 
\end{theorem}

One consequence of Theorem \ref{Thm:intro-2} is the following corollary.
\begin{corollary} (\cite{MQ22g}) \label{Cor:intro-1}
For $n\geq 3$, let $\Omega$ be a domain in the standard unit round sphere $(\mathbb{S}^n, g_{\mathbb{S}})$.
Suppose that, on $\Omega$, there is a complete conformal metric $g= e^{2u}g_{\mathbb{S}}$ satisfying either $Ric_g$ is nonnegative outside
a compact subset or 
\begin{enumerate}
\item $Ric^{-}_g \in L^{1}(\Omega, g) \cap L^\infty(\Omega, g)$
\item $R_g \in L^{\infty}(\Omega, g) \text{ and } |\nabla^g R_g|\in L^{\infty}(\Omega, g)$. 
\end{enumerate}
Then $\partial\Omega = \mathbb{S}^n\setminus\Omega$ is a finite point set. 
\end{corollary}


\subsection{Applications to fully nonlinear elliptic equations}\label{Subsec:PDE}

In this subsection we collect some corollaries of Theorem \ref{Thm:main-asymptotic-p} on the solutions to fully nonlinear elliptic equations.
It is interesting to compare the intermediate positivity cones $\mathcal{A}^{(p)}$ with those in the study of fully nonlinear 
equations. Recall
$$
\Gamma^k = \{\lambda = \{(\lambda_1, \lambda_2, \cdots, \lambda_n)\in \mathbb{R}^n: \sigma_1(\lambda)\geq 0, \sigma_2(\lambda)\geq 0, \cdots, \sigma_k(\lambda)\geq 0\}
$$
for $k=1, 2, \cdots, n$, where $\sigma_l$ is the elementary symmetric functions for $l=1, 2, \cdots, n$. It is easily seen that $\mathcal{A}^{(2)} =\Gamma^1$ 
and $\mathcal{A}^{(p)}$ approaches $\Gamma^n$ as $p\to\infty$. Hence, for any positive cone $\Gamma$ between $\Gamma^1$ and 
$\Gamma^n$, we may consider
\begin{equation}\label{Equ:p-index}
p_\Gamma = \max \{p: \Gamma \subset \mathcal{A}^{(p)}\}.
\end{equation}
$p_\Gamma$ is useful when one uses $p$-superharmonic functions to study solutions to a class of fully nonlinear elliptic equations. We first realize

\begin{lemma}\label{Lem:fully-p-gamma} (\cite{LMQZ1, LMQZ2}) 
Suppose that $u$ is nonnegative and that $u\in C^2 (\Omega\setminus S)$ for a compact subset $S$ of a
bounded domain $\Omega$ in $\mathbb{R}^n$. And suppose $\lim_{x\to S}u(x) = +\infty.$ 
Assume -$\lambda (D^2 u(x)) \in \Gamma$ for $p_\Gamma \in (1, n]$. Then $u$ is a $p_\Gamma$-superharmonic function in $\Omega$.
\end{lemma}

We remark that, from the proof of \cite[Lemma 3.2 and 3.3]{MQ21a} (see also \cite[Proposition 1.1]{BV89} when $S$ is an isolated point), 
it is easily seen that $-\Delta_{p_\Gamma} u$ is a Radon measure under even somewhat weaker assumptions. Consequently, we have
 
\begin{corollary}\label{Cor:app-fully-nonlinear} (\cite{LMQZ1, LMQZ2})
Suppose that $u$ is nonnegative and that $u\in C^2 (\Omega\setminus S)$ for a compact subset $S$ of a bounded domain $\Omega$ in $\mathbb{R}^n$. 
And suppose $\lim_{x\to S}u(x) = +\infty.$ 
Assume -$\lambda (D^2 u(x)) \in \Gamma$ for $p_\Gamma \in (1, n]$. Then $S$ is of Hausdorff dimension not greater than $n-p_\Gamma$ and, 
for $x_0\in S$, there are a subset $E$ that is $p_\Gamma$-thin 
for the singular behavior at $x_0$ and a nonnegative number $m$ such that
$$
\lim_{x\to x_0 \text{ and } x\notin E}  \frac {u(x)}{G_{p_\Gamma}(x, x_0)}= m.
$$
Moreover $u(x) \geq m G_{p_\Gamma}(x, x_0) - c_0$ in some neighborhood of $x_0$.
\end{corollary}

It seems surprising that 
we have a rather effective way to calculate $p_\Gamma$ for a cone associated with a homogeneous, symmetric, convex function of $n$-variables.

\begin{lemma}\label{Lem:calculate-p-index} (\cite{LMQZ1, LMQZ2}) Suppose that $\Gamma$ is a cone given by a homogeneous, symmetric, convex function $F(\lambda)$ on 
$\mathbb{R}^n$. Let $(-\frac {n-1}{p-1}, 1, 1, \cdots, 1)\in \partial\Gamma = \{\lambda\in \mathbb{R}^n: F(\lambda)=0\}$. Then 
$\Gamma\subset\mathcal{A}^{(p)}$ and $p_\Gamma = p$.
\end{lemma}

Consequently, we can calculate $p_{\Gamma^k}$ easily. 

\begin{corollary}\label{Cor:gamma-k-p-index}  (\cite{LMQZ1, LMQZ2}) For the positive cone $\Gamma^k$, we have
\begin{equation}\label{Equ:gamma-k-p-index}
p_{\Gamma^k} = \frac {n(k-1)}{n-k} + 2 \in [2, n]
\end{equation}
for $1\leq k \leq \frac n2$.
\end{corollary}
\begin{proof} This simply is because
$$
\sigma_k(-\frac {n-k}k, 1, 1, \cdots, 1) = -\frac {n-k}k \left(\aligned k-1\\n-1\endaligned\right) + \left(\aligned & k\\n & -1\endaligned\right) = 0.
$$
\end{proof}  
 
Remarkably, we are able to derive an asymptotic estimates that extends \cite[Theorem 3.6]{Lab02} significantly.

\begin{corollary}\label{Cor:sigma-k-asymptotic} (\cite{LMQZ1, LMQZ2}) Suppose that $u$ is nonnegative and that $u\in C^2 (\Omega\setminus S)$ for a
compact subset $S$ inside a bounded domain $\Omega$ in $\mathbb{R}^n$. And suppose 
$\lim_{x\to S}u(x) = +\infty.$
Assume $-\lambda (D^2 u(x)) \in \Gamma^k$ for $ 1\leq k \leq \frac n2$. Then $S$ is of Hausdorff dimension not greater than $n-p_\Gamma$ and, 
for $x_0\in S$, there are a subset $E$ that is $p_{\Gamma^k}$-thin 
for the singular behavior at $x_0$ and a nonnegative number $m$ such that
$$
\lim_{x\to x_0 \text{ and } x\notin E}  \frac {u(x)}{\mathcal{G}^k(x, x_0)} = m.
$$
Moreover $u(x) \geq m \mathcal{G}^k(x, x_0) - c_0$ in some neighborhood of $x_0$, where
$$
\mathcal{G}^k(x, x_0) = \left\{\aligned |x-x_0|^{2 - \frac nk} & \text{ when } 1\leq k < \frac n2\\
- \log |x-x_0| & \text{ when } k = \frac n2\endaligned\right..
$$
\end{corollary}

In general it takes a lot more to rule out the thin set $E$ in the above two corollaries, even for isolated singularities (cf. \cite{KV, KV-e,Vn17, Lab02}).


\vskip 0.3cm
\noindent$\mbox{}^\dag$School of Mathematical Science and LPMC, Nankai University, Tianjin, China; \\e-mail: 
msgdyx8741@nankai.edu.cn 
\vspace{0.2cm}

\noindent $\mbox{}^\ddag$ Department of Mathematics, University of California, Santa Cruz, CA 95064; \\
e-mail: qing@ucsc.edu

\end{document}